\documentclass[11pt]{article}
\usepackage[margin=1in]{geometry}
\usepackage[T1]{fontenc}
\usepackage[utf8]{inputenc}
\usepackage{lmodern}
\usepackage{amsmath,amssymb,amsthm,mathtools}
\usepackage{bm}
\usepackage{physics}
\usepackage{enumitem}
\usepackage{hyperref}
\usepackage{xcolor}
\usepackage{tikz}
\usepackage{microtype}
\usepackage{mathrsfs}
\usepackage{authblk} 
\usepackage[numbers,sort&compress]{natbib} 

\hypersetup{
  colorlinks=true,
  linkcolor=blue!60!black,
  citecolor=blue!60!black,
  urlcolor=blue!60!black
}

\newtheorem{theorem}{Theorem}
\newtheorem{proposition}{Proposition}

\theoremstyle{definition}
\newtheorem{definition}{Definition}
\newtheorem{example}{Example}
\newtheorem{remark}{Remark}

\newcommand{\R}{\mathbb{R}}
\newcommand{\x}{\bm{x}}
\newcommand{\uvec}{\bm{u}}
\newcommand{\lam}{\bm{\lambda}}
\newcommand{\gvec}{\bm{g}}
\newcommand{\argmax}{\operatorname*{arg\,max}}
\newcommand{\argmin}{\operatorname*{arg\,min}}
\newcommand{\clamp}{\operatorname{clamp}}

\title{\textbf{Lecture Note for Bounded Controls in Continuous-Time and Control of Several Variables}}
\author{Louis Shuo Wang\thanks{Email: wang.s41@northeastern.edu}}
\affil{Department of Mathematics, Northeastern University, Boston, MA 02115, USA}

\date{Prepared for senior undergraduate Students}

\begin{document}
\maketitle

\begin{abstract}
In this note, we develop the first-order theory of optimal control problems with box constraints on the control. We emphasize the precise modification of Pontryagin's maximum principle when the admissible control set is compact, the projection/clamping formula for scalar quadratic Hamiltonians, the distinction between \emph{intrinsic projection inside the optimality system} and \emph{post hoc truncation of an unconstrained solution}, and the corresponding forward-backward sweep implementation. The presentation is pitched at senior PhD students who are already comfortable with variational arguments, adjoint systems, and basic nonlinear analysis. \textbf{These notes are mainly based on the book ``optimal control applied to biological models'' of Suzanne Lenhart and John T. Workman.}
\end{abstract}

\tableofcontents

\section{Part I}
\subsection{Motivation and problem class}

In many scientifically relevant control models, the control cannot range freely over all of $\R^m$. Physical, technological, pharmacological, ecological, or regulatory considerations impose pointwise bounds:
\[
u(t)\in U \subset \R^m,
\qquad
U = [a_1,b_1]\times \cdots \times [a_m,b_m].
\]
Examples include:
\begin{enumerate}[label=(\roman*)]
    \item dosage or treatment intensity: $0\le u_i(t)\le u_i^{\max}$,
    \item harvesting effort or extraction rate: $0\le u(t)\le \bar u$,
    \item actuator saturation in engineering: $\|u(t)\|_\infty \le \bar u$.
\end{enumerate}

A prototypical Bolza problem is
\begin{equation}
\label{eq:bolza}
\max_{u(\cdot)}\;
J(u)
=
\int_{t_0}^{t_1} f(t,x(t),u(t))\,dt + \phi(x(t_1))
\end{equation}
subject to
\begin{equation}
\label{eq:dynamics}
\dot x(t)=g(t,x(t),u(t)),
\qquad
x(t_0)=x_0,
\qquad
u(t)\in U \text{ a.e. on } [t_0,t_1].
\end{equation}
Throughout, $x:[t_0,t_1]\to \R^n$ is the state and $u:[t_0,t_1]\to \R^m$ is the control.

\medskip

The key message is simple but fundamental:

\begin{center}
\emph{For constrained controls, one does not merely solve the unconstrained stationarity equation and then clip the answer afterward. The control constraint must be enforced \textbf{inside} the optimality system.}
\end{center}

\subsection{Hamiltonian formalism with box constraints}

\begin{definition}[Hamiltonian]
For \eqref{eq:bolza}--\eqref{eq:dynamics}, define the Hamiltonian
\[
H(t,x,u,\lambda) := f(t,x,u) + \lambda^\top g(t,x,u),
\]
where $\lambda(\cdot)\in \R^n$ is the adjoint (costate).
\end{definition}

The unconstrained first-order condition often takes the form
\[
H_u(t,x^*,u^*,\lambda^*) = 0.
\]
With box constraints, this is generally false on intervals where the optimal control lies on the boundary of the admissible set.

\subsubsection{Constrained Pontryagin principle}

For a maximization problem, Pontryagin's principle states that the optimal control must maximize the Hamiltonian \emph{pointwise over the admissible control set}:
\begin{equation}
\label{eq:pmp-max}
u^*(t)\in \argmax_{v\in U} H(t,x^*(t),v,\lambda^*(t))
\qquad \text{for a.e. } t\in[t_0,t_1].
\end{equation}
The state and adjoint equations remain
\begin{align}
\dot x^*(t) &= H_\lambda(t,x^*(t),u^*(t),\lambda^*(t)),
\label{eq:state-pmp}
\\
\dot \lambda^*(t) &= - H_x(t,x^*(t),u^*(t),\lambda^*(t)),
\label{eq:adjoint-pmp}
\end{align}
with transversality
\[
\lambda^*(t_1)=\nabla \phi(x^*(t_1))
\]
when the terminal state is free and the final time is fixed.

\begin{remark}
For minimization problems, the same equations \eqref{eq:state-pmp}--\eqref{eq:adjoint-pmp} hold, but \eqref{eq:pmp-max} is replaced by pointwise minimization:
\[
u^*(t)\in \argmin_{v\in U} H(t,x^*(t),v,\lambda^*(t)).
\]
\end{remark}

\subsection{Scalar box constraint: complementary sign conditions}

Assume a scalar control $u\in[a,b]$ and $H$ is differentiable in $u$.
Then the constrained maximum condition is equivalent to the familiar sign logic:
\begin{equation}
\label{eq:sign-conditions-max}
\begin{cases}
H_u(t,x^*,u^*,\lambda^*) < 0 \;\Rightarrow\; u^*(t)=a,\\[0.25em]
H_u(t,x^*,u^*,\lambda^*) = 0 \;\Rightarrow\; a\le u^*(t)\le b,\\[0.25em]
H_u(t,x^*,u^*,\lambda^*) > 0 \;\Rightarrow\; u^*(t)=b.
\end{cases}
\end{equation}
Equivalently,
\begin{equation}
\label{eq:complementary-max}
\begin{cases}
u^*(t)=a \;\Rightarrow\; H_u(t,x^*,u^*,\lambda^*)\le 0,\\[0.25em]
a<u^*(t)<b \;\Rightarrow\; H_u(t,x^*,u^*,\lambda^*)=0,\\[0.25em]
u^*(t)=b \;\Rightarrow\; H_u(t,x^*,u^*,\lambda^*)\ge 0.
\end{cases}
\end{equation}

For minimization, the inequalities reverse:
\begin{equation}
\label{eq:sign-conditions-min}
\begin{cases}
H_u > 0 \;\Rightarrow\; u^*=a,\\
H_u = 0 \;\Rightarrow\; a\le u^*\le b,\\
H_u < 0 \;\Rightarrow\; u^*=b.
\end{cases}
\end{equation}

\subsubsection{Variational derivation in one paragraph}

Let $u^*$ be optimal and let $u_\varepsilon=u^*+\varepsilon h$ be an admissible one-sided perturbation with $\varepsilon>0$ small. The first variation satisfies
\[
0 \ge \frac{d}{d\varepsilon}J(u_\varepsilon)\bigg|_{\varepsilon=0^+}
= \int_{t_0}^{t_1} H_u(t,x^*,u^*,\lambda^*)\, h(t)\,dt.
\]
At a time $s$ where $u^*(s)<b$, one may choose a nonnegative bump $h$ concentrated near $s$. Therefore $H_u(s)\le 0$. Similarly, if $u^*(s)>a$, one may choose a nonpositive bump, yielding $H_u(s)\ge 0$. Combined, these imply \eqref{eq:complementary-max}. This is the one-dimensional box-constrained version of the Hamiltonian maximization condition.

\subsection{Projection formula for quadratic Hamiltonians}

A large class of biological, economic, and engineering models leads to Hamiltonians that are strictly concave or convex in $u$. For a scalar maximization problem, suppose
\[
H(t,x,u,\lambda)
=
\Psi(t,x,\lambda) + \beta(t,x,\lambda)\,u - \frac{\rho}{2}u^2,
\qquad \rho>0.
\]
Then the unconstrained maximizer is
\[
\tilde u(t)=\frac{\beta(t,x,\lambda)}{\rho}.
\]
Imposing $u\in[a,b]$ gives the exact constrained maximizer
\begin{equation}
\label{eq:projection}
u^*(t)
=
\clamp_{[a,b]}\!\left(\tilde u(t)\right)
=
\min\{b,\max\{a,\tilde u(t)\}\}.
\end{equation}

\begin{proposition}[Projection formula]
If $H(t,x,\cdot,\lambda)$ is strictly concave and quadratic in $u$, then \eqref{eq:projection} is the unique maximizer of $H$ over $[a,b]$.
\end{proposition}

\begin{proof}
Since $H$ is strictly concave in $u$, any local maximizer is the unique global maximizer. The unconstrained critical point $\tilde u$ is unique. If $\tilde u\in[a,b]$, it is feasible and optimal. If $\tilde u<a$, strict concavity implies $H$ is increasing up to $\tilde u$ and decreasing thereafter, so the constrained maximizer is the left endpoint $a$. The case $\tilde u>b$ is analogous.
\end{proof}

\begin{remark}
The ``projection'' in \eqref{eq:projection} is not a heuristic numerical correction. It is the exact pointwise optimizer of the Hamiltonian over the admissible interval.
\end{remark}

\subsection{Bang-bang structure and switching functions}

When the Hamiltonian is affine in the control,
\[
H(t,x,u,\lambda)
=
\Psi(t,x,\lambda) + \sigma(t)\,u,
\qquad u\in[a,b],
\]
the pointwise maximization problem is trivial:
\begin{equation}
\label{eq:bang-bang}
u^*(t)=
\begin{cases}
b, & \sigma(t)>0,\\
[a,b], & \sigma(t)=0,\\
a, & \sigma(t)<0.
\end{cases}
\end{equation}
The scalar function $\sigma(t):=H_u(t,x^*,u^*,\lambda^*)$ is called the \emph{switching function}. If its zeros are isolated, the optimal control is bang-bang: it sits on one bound and switches to the other at isolated times.

\begin{remark}[Singular arcs]
If $\sigma(t)\equiv 0$ on a nontrivial interval, first-order PMP does not determine $u^*$ there. One must differentiate the switching function until the control enters, or invoke higher-order optimality conditions such as Legendre--Clebsch type criteria.
\end{remark}

\subsection{Worked example: linear state dynamics, quadratic running cost}

Consider the maximization problem
\begin{equation}
\label{eq:ex-problem}
\max_{0\le u\le 2}\;
J(u)
=
\int_0^T \big(\alpha x(t)-\tfrac12 r u(t)^2\big)\,dt
\end{equation}
subject to
\[
\dot x(t)=a x(t)+b u(t),\qquad x(0)=x_0,
\qquad \alpha,r>0.
\]

\subsubsection*{Hamiltonian}
\[
H(x,u,\lambda)=\alpha x - \frac12 r u^2 + \lambda(ax+bu).
\]

\subsubsection*{Adjoint equation}
\[
\dot\lambda(t)=-H_x=-(\alpha+a\lambda(t)),
\qquad
\lambda(T)=0.
\]
Thus
\[
\lambda(t)=
\begin{cases}
\displaystyle \frac{\alpha}{a}\Big(e^{a(T-t)}-1\Big), & a\neq 0,\\[0.75em]
\alpha(T-t), & a=0.
\end{cases}
\]

\subsubsection*{Control characterization}
\[
H_u = -r u + b\lambda.
\]
The unconstrained maximizer is
\[
\tilde u(t)=\frac{b}{r}\lambda(t).
\]
Hence the constrained optimizer is
\[
u^*(t)=\min\left\{2,\max\left\{0,\frac{b}{r}\lambda(t)\right\}\right\}.
\]

\subsubsection*{Detailed Interpretation of the Control Structure}

The optimal constrained control $u^*(t)$ is strictly determined by the exact projection of the unconstrained ideal control $\tilde u(t) = \frac{b}{r}\lambda(t)$ onto the admissible interval $[0, 2]$. Because the adjoint variable carries a terminal condition $\lambda(T)=0$ and is typically strictly positive and monotonically decaying for $t < T$, the control trajectory naturally partitions into three distinct phases. This yields the canonical \emph{bound--interior--bound} structure:

\begin{enumerate}[label=(\roman*)]
    \item \textbf{Initial Upper Bound (Saturation):} Early in the control horizon ($t \ll T$), the future marginal value of the state is high, making $\lambda(t)$ sufficiently large such that $\tilde u(t) > 2$. The physical constraints of the system become active, and the optimal action is to saturate the input: $u^*(t) = 2$.
    \item \textbf{Interior Arc (Unconstrained Operation):} As time progresses toward the terminal horizon, the value of the adjoint $\lambda(t)$ diminishes. Once the ideal control drops into the feasible region ($0 < \tilde u(t) < 2$), the bounds become inactive. The optimal control strictly follows the continuous, unconstrained feedback law: $u^*(t) = \tilde u(t)$.
    \item \textbf{Terminal Lower Bound (Decay to Zero):} Approaching the terminal time $T$, the residual value of the state drops to zero ($\lambda(T)=0$). Consequently, the ideal control $\tilde u(t)$ decays entirely. The optimal control smoothly tracks this decay, hitting the lower bound $0$ at exactly $t=T$ (or earlier, if $\lambda(t)$ were to become non-positive).
\end{enumerate}

The following diagram geometrically illustrates this projection mechanism, highlighting why bounded optimal control necessitates a piecewise analytical structure rather than a simple post hoc truncation of the final trajectory.

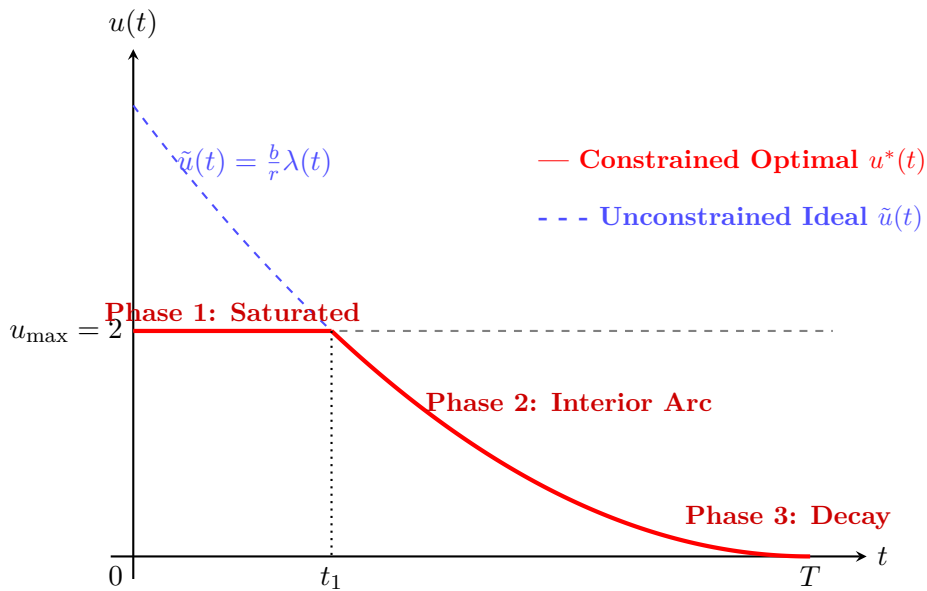
\begin{figure}[htbp]
\centering
\begin{tikzpicture}[>=stealth, scale=1.5]
    \draw[->, thick] (-0.2, 0) -- (6.5, 0) node[right] {$t$};
    \draw[->, thick] (0, -0.2) -- (0, 4.5) node[above] {$u(t)$};
    \node[below left] at (0,0) {$0$};

    \draw[dashed, gray, thick] (0, 2) node[left, text=black] {$u_{\max} = 2$} -- (6.2, 2);

    \draw[blue!70, thick, dashed, domain=0:6, samples=100] 
        plot (\x, {4*(1-\x/6)^2});
        
    \node[text=blue!70, right] at (0.3, 3.5) {$\tilde{u}(t) = \frac{b}{r}\lambda(t)$};

    \def\tcut{1.757}

    \draw[red, ultra thick] (0, 2) -- (\tcut, 2);
    \draw[red, ultra thick, domain=\tcut:6, samples=100] 
        plot (\x, {4*(1-\x/6)^2});

    \draw[dotted, thick] (\tcut, 2) -- (\tcut, 0) node[below] {$t_1$};
    \node[below] at (6, 0) {$T$};

    \node[above, text=red!80!black, font=\small\bfseries] at (\tcut/2, 2) {Phase 1: Saturated};
    \node[above right, text=red!80!black, font=\small\bfseries] at (2.5, 1.2) {Phase 2: Interior Arc};
    \node[above right, text=red!80!black, font=\small\bfseries] at (4.8, 0.15) {Phase 3: Decay};

    \node[anchor=west, align=left, text=red, font=\small\bfseries] at (3.5, 3.5) {--- Constrained Optimal $u^*(t)$};
    \node[anchor=west, align=left, text=blue!70, font=\small\bfseries] at (3.5, 3.0) {- - - Unconstrained Ideal $\tilde{u}(t)$};
\end{tikzpicture}
\caption{The projection formula in action. The unconstrained ideal control $\tilde{u}(t)$ (driven by the adjoint $\lambda(t)$) is dynamically clamped to the $[0, 2]$ interval, producing the piece-wise bound--interior--bound optimal trajectory.}
\label{fig:bound-interior-bound}
\end{figure}

\subsection{Why post hoc truncation is generally wrong}

A frequent mistake in applied work is:
\begin{enumerate}[label=(\roman*), leftmargin=2em]
    \item solve the unconstrained optimality system,
    \item obtain $\hat u$ that violates the bounds,
    \item define $u^{\text{clip}}(t)=\clamp_{[a,b]}(\hat u(t))$ afterward.
\end{enumerate}
This is generally \emph{not} the optimal constrained control.

\begin{remark}[Structural reason]
The state and adjoint satisfy
\[
\dot x = g(t,x,u),\qquad \dot\lambda = -H_x(t,x,u,\lambda),
\]
so changing the control changes both $x$ and $\lambda$, which in turn changes the control law. The constrained problem has a different coupled boundary-value structure from the unconstrained problem. Projection must therefore occur \emph{within each iterate of the solve}, not after the fact.
\end{remark}

\subsubsection*{A useful slogan}
\begin{center}
\emph{Project the control law, not the finished trajectory.}
\end{center}

\subsection{Forward-backward sweep with box constraints}

The forward-backward sweep method (FBSM) is a standard indirect numerical scheme for solving PMP systems.

Let $u^{(k)}$ be the current control iterate.

\begin{enumerate}[label=(\roman*)]
    \item \textbf{Forward state solve:}
    solve
    \[
    \dot x^{(k)}(t)=g\big(t,x^{(k)}(t),u^{(k)}(t)\big),
    \qquad x^{(k)}(t_0)=x_0.
    \]

    \item \textbf{Backward adjoint solve:}
    solve
    \[
    \dot\lambda^{(k)}(t)= -H_x\big(t,x^{(k)}(t),u^{(k)}(t),\lambda^{(k)}(t)\big),
    \qquad \lambda^{(k)}(t_1)=\nabla\phi(x^{(k)}(t_1)).
    \]

    \item \textbf{Control update:}
    compute the pointwise optimizer
    \[
    \tilde u^{(k+1)}(t)
    \in
    \argmax_{v\in U}
    H\big(t,x^{(k)}(t),v,\lambda^{(k)}(t)\big),
    \]
    and then often damp it:
    \[
    u^{(k+1)}(t)=\theta\, \tilde u^{(k+1)}(t)+(1-\theta)u^{(k)}(t),
    \qquad 0<\theta\le 1.
    \]
\end{enumerate}

If the unconstrained scalar update is $\tilde u = \Gamma(t,x,\lambda)$ and the admissible interval is $[a,b]$, then the bounded update is
\[
\tilde u^{(k+1)}(t)=\min\{b,\max\{a,\Gamma(t,x^{(k)},\lambda^{(k)})\}\}.
\]
This is the exact mathematical content of the constrained maximization step.

\begin{remark}
In practice, damping is often essential for stability, especially for stiff dynamics, long horizons, or highly nonlinear state-costate coupling.
\end{remark}

\subsection{Multi-input controls and normal cones}

For $u\in U\subset \R^m$ with $U$ closed and convex, the pointwise maximization condition can be expressed as
\[
0 \in -H_u(t,x^*,u^*,\lambda^*) + N_U(u^*(t)),
\]
where $N_U(u)$ is the normal cone to $U$ at $u$. This is the natural variational inequality form of the control law. For a box
\[
U=[a_1,b_1]\times \cdots \times [a_m,b_m],
\]
the control law decouples componentwise whenever $H$ is separable in the control coordinates.

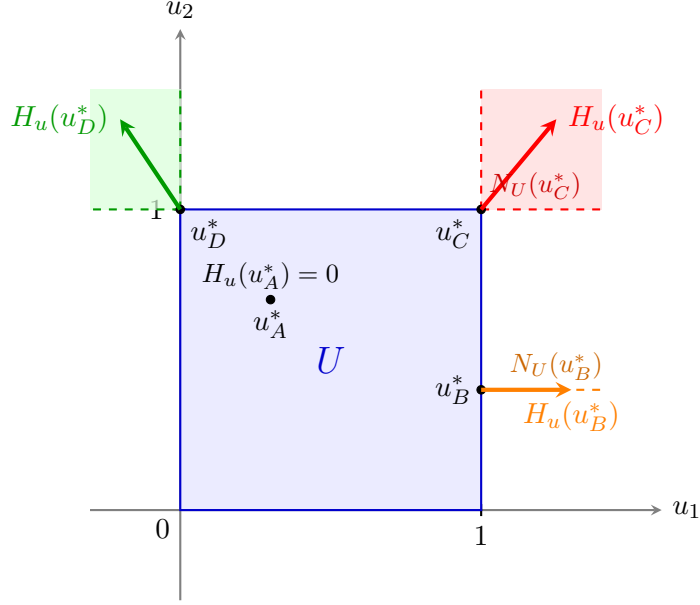
\begin{figure}[htbp]
\centering
\begin{tikzpicture}[scale=4, >=stealth]
    \draw[->, thick, gray] (-0.3, 0) -- (1.6, 0) node[right, text=black] {$u_1$};
    \draw[->, thick, gray] (0, -0.3) -- (0, 1.6) node[above, text=black] {$u_2$};

    \draw[thick] (1, 0.02) -- (1, -0.02) node[below] {$1$};
    \draw[thick] (0.02, 1) -- (-0.02, 1) node[left] {$1$};
    \node[below left] at (0,0) {$0$};

    \fill[blue!8] (0,0) rectangle (1,1);
    \draw[blue!80!black, thick] (0,0) rectangle (1,1);
    \node[blue!80!black, font=\Large\bfseries] at (0.5, 0.5) {$U$};

    \fill[red!15, opacity=0.7] (1,1) rectangle (1.4, 1.4);
    \draw[red, dashed, thick] (1,1) -- (1.4, 1);
    \draw[red, dashed, thick] (1,1) -- (1, 1.4);
    
    \filldraw[black] (1, 1) circle (0.4pt) node[below left] {$u^*_C$};
    \draw[->, ultra thick, red] (1,1) -- (1.25, 1.3) node[right] {$H_u(u^*_C)$};
    \node[red!80!black, font=\small] at (1.18, 1.08) {$N_U(u^*_C)$};

    \draw[orange, dashed, thick] (1,0.4) -- (1.4, 0.4);
    \node[orange!80!black, font=\small, above] at (1.25, 0.4) {$N_U(u^*_B)$};
    
    \filldraw[black] (1, 0.4) circle (0.4pt) node[left] {$u^*_B$};
    \draw[->, ultra thick, orange] (1,0.4) -- (1.3, 0.4) node[below] {$H_u(u^*_B)$};

    \filldraw[black] (0.3, 0.7) circle (0.4pt) node[below] {$u^*_A$};
    \node[above, font=\small, text=black] at (0.3, 0.7) {$H_u(u^*_A) = 0$};
    
    \fill[green!15, opacity=0.7] (0,1) rectangle (-0.3, 1.4);
    \draw[green!60!black, dashed, thick] (0,1) -- (-0.3, 1);
    \draw[green!60!black, dashed, thick] (0,1) -- (0, 1.4);
    
    \filldraw[black] (0, 1) circle (0.4pt) node[below right] {$u^*_D$};
    \draw[->, ultra thick, green!60!black] (0,1) -- (-0.2, 1.3) node[left] {$H_u(u^*_D)$};

\end{tikzpicture}
\caption{Geometric interpretation of the variational inequality $H_u(u^*) \in N_U(u^*)$ for a 2D box constraint $U=[0,1]^2$. 
\textbf{(Interior $u^*_A$):} The normal cone is $\{0\}$, implying stationarity $H_u=0$. 
\textbf{(Edge $u^*_B$):} The control $u_1$ is saturated at $1$, so the gradient $H_{u_1}$ must point outwards (positive), while $H_{u_2}=0$. 
\textbf{(Corners $u^*_C, u^*_D$):} Multiple constraints are active. At $u^*_C=(1,1)$, the normal cone is the first quadrant (shifted), meaning both components of the gradient $H_u$ must be non-negative.}
\label{fig:normal-cone}
\end{figure}

\subsection{A compact theorem statement}

\begin{theorem}[PMP with box-constrained scalar control]
Consider the maximization problem
\[
\max_{u(\cdot)} \int_{t_0}^{t_1} f(t,x(t),u(t))\,dt + \phi(x(t_1))
\]
subject to
\[
\dot x(t)=g(t,x(t),u(t)),\qquad x(t_0)=x_0,\qquad a\le u(t)\le b.
\]
Assume regularity sufficient for Pontryagin's principle. If $(x^*,u^*)$ is optimal, then there exists an absolutely continuous adjoint $\lambda$ such that
\[
\dot x^* = H_\lambda(t,x^*,u^*,\lambda),\qquad
\dot\lambda = -H_x(t,x^*,u^*,\lambda),
\qquad
\lambda(t_1)=\nabla\phi(x^*(t_1)),
\]
and
\[
u^*(t)\in \argmax_{v\in[a,b]} H(t,x^*(t),v,\lambda(t))
\quad \text{for a.e. } t.
\]
If $H$ is differentiable with respect to $u$, then for a.e. $t$,
\[
u^*(t)=a \Rightarrow H_u\le 0,\qquad
a<u^*(t)<b \Rightarrow H_u=0,\qquad
u^*(t)=b \Rightarrow H_u\ge 0.
\]
\end{theorem}

\section{Part II}

\subsection{Motivation and Scope}

Many modern control problems are intrinsically multi-dimensional. In mathematical biology one tracks several interacting populations; in epidemiology one may control vaccination, treatment, and isolation simultaneously; in engineering one often regulates a vector-valued state with multiple actuators. The scalar one-state/one-control formalism is therefore only a warm-up. The correct framework is vector-valued state dynamics together with vector-valued controls, and the natural analytical object is the Hamiltonian associated with the control system.

In Part II we develop four themes:

\begin{enumerate}[label=(\roman*)]
    \item Pontryagin necessary conditions for problems with multiple states and controls.
    \item The finite-horizon continuous-time linear quadratic regulator (LQR) and the Riccati differential equation.
    \item Reduction of higher-order differential equations to first-order systems.
    \item Isoperimetric constraints and their conversion into endpoint constraints via an auxiliary state.
\end{enumerate}

The presentation assumes familiarity with ordinary differential equations, functional analysis at the level of weak compactness heuristics, and basic optimal control.

\subsection{General Multi-State, Multi-Control Optimal Control Problems}

Let the state be $\x(t)\in\R^n$ and the control be $\uvec(t)\in\R^m$. Consider the finite-horizon problem
\begin{equation}
\max_{\uvec(\cdot)} \left\{
J(\uvec):=\int_{t_0}^{t_1} f(t,\x(t),\uvec(t))\,dt + \phi(\x(t_1))
\right\}
\end{equation}
subject to
\begin{equation}
\dot{\x}(t)=\gvec(t,\x(t),\uvec(t)), \qquad \x(t_0)=\x_0.
\end{equation}
Here
\[
f:[t_0,t_1]\times\R^n\times\R^m\to\R,\qquad
\gvec:[t_0,t_1]\times\R^n\times\R^m\to\R^n,
\]
and $\phi:\R^n\to\R$ are assumed continuously differentiable unless stated otherwise.

\subsubsection{Hamiltonian Formalism}

Introduce the adjoint (costate) vector $\lam(t)\in\R^n$. The Hamiltonian is
\begin{equation}
H(t,\x,\uvec,\lam)=f(t,\x,\uvec)+\lam^\top \gvec(t,\x,\uvec).
\end{equation}

This compactly encodes both the dynamics and the objective.

\subsubsection{Pontryagin Necessary Conditions}

\begin{theorem}[Pontryagin principle, smooth finite-horizon form]
Suppose $\uvec^*$ is an optimal control with corresponding state $\x^*$. Then there exists an absolutely continuous adjoint $\lam:[t_0,t_1]\to\R^n$ such that, for almost every $t\in[t_0,t_1]$,
\begin{align}
\dot{\x}^*(t) &= \pdv{H}{\lam}(t,\x^*(t),\uvec^*(t),\lam(t))
= \gvec(t,\x^*(t),\uvec^*(t)), \\
\dot{\lam}(t) &= -\pdv{H}{\x}(t,\x^*(t),\uvec^*(t),\lam(t)), \\
\uvec^*(t) &\in \argmax_{\uvec\in U} H(t,\x^*(t),\uvec,\lam(t)),
\end{align}
with terminal condition
\begin{equation}
\lam(t_1)=\nabla \phi(\x^*(t_1)),
\end{equation}
provided the terminal state is free and $U\subseteq\R^m$ is the admissible control set.
\end{theorem}

\begin{remark}
For minimization problems, one can either minimize the Hamiltonian directly or maximize the negative Hamiltonian. Both conventions are used in the literature. What matters is consistency.
\end{remark}

\subsubsection{Componentwise Form}

If $\x=(x_1,\dots,x_n)^\top$, $\uvec=(u_1,\dots,u_m)^\top$, and $\lam=(\lambda_1,\dots,\lambda_n)^\top$, then
\[
H(t,\x,\uvec,\lam)
=
f(t,\x,\uvec)+\sum_{i=1}^{n}\lambda_i\,g_i(t,\x,\uvec),
\]
and the necessary conditions become
\begin{align}
\dot{x}_i &= \pdv{H}{\lambda_i}=g_i(t,\x,\uvec), \qquad i=1,\dots,n, \\
\dot{\lambda}_j &= -\pdv{H}{x_j}, \qquad j=1,\dots,n, \\
0 &= \pdv{H}{u_k} \quad \text{if } u_k \text{ is unconstrained and the optimum is interior}.
\end{align}

\subsubsection{Control Bounds}

If a scalar control component satisfies
\[
a_k \le u_k(t)\le b_k,
\]
then the stationarity condition is replaced by the projection rule
\[
u_k^*(t)=
\begin{cases}
a_k, & \text{if } \displaystyle \pdv{H}{u_k}(t,\x^*,\uvec^*,\lam)<0,\\[1ex]
\text{interior value solving } \displaystyle \pdv{H}{u_k}=0, & \text{if admissible},\\[1ex]
b_k, & \text{if } \displaystyle \pdv{H}{u_k}(t,\x^*,\uvec^*,\lam)>0.
\end{cases}
\]
This is the smooth analog of a bang-bang/boundary-interior structure.

\subsubsection{Endpoint Variants}

The transversality conditions depend on which endpoint values are fixed or free:

\begin{enumerate}[label=(\roman*)]
    \item If $x_i(t_1)$ is free and $\phi$ is present, then $\lambda_i(t_1)=\partial \phi/\partial x_i$.
    \item If $x_i(t_1)$ is fixed, then no transversality condition is imposed on $\lambda_i(t_1)$.
    \item If there is no terminal payoff, then $\phi\equiv 0$ and free terminal states imply $\lambda_i(t_1)=0$.
\end{enumerate}

\subsection{A Worked Multi-State Example}

Consider the minimization problem
\begin{equation}
\min_{u(\cdot)} \int_0^1 \left( x_2(t)+u(t)^2 \right)\,dt
\end{equation}
subject to
\begin{align}
\dot{x}_1(t) &= x_2(t), \qquad x_1(0)=0,\quad x_1(1)=1,\\
\dot{x}_2(t) &= u(t), \qquad x_2(0)=0.
\end{align}

\subsubsection{Hamiltonian and Adjoint System}

For minimization, use
\[
H=x_2+u^2+\lambda_1 x_2+\lambda_2 u.
\]
The adjoint equations are
\begin{align}
\dot{\lambda}_1 &= -\pdv{H}{x_1}=0, \\
\dot{\lambda}_2 &= -\pdv{H}{x_2}=-(1+\lambda_1).
\end{align}
Hence
\[
\lambda_1(t)\equiv C, \qquad
\lambda_2(t)=-(C+1)(t-1),
\]
where the terminal condition $\lambda_2(1)=0$ was used.

The stationarity condition gives
\[
0=\pdv{H}{u}=2u+\lambda_2
\quad \Longrightarrow \quad
u^*(t)=-\frac{\lambda_2(t)}{2}=\frac{C+1}{2}(t-1).
\]

\subsubsection{State Recovery}

Then
\[
\dot{x}_2(t)=\frac{C+1}{2}(t-1)
\]
implies
\[
x_2(t)=\frac{C+1}{2}\left(\frac{t^2}{2}-t\right),
\]
using $x_2(0)=0$. Next,
\[
\dot{x}_1=x_2
\]
gives
\[
x_1(t)=\frac{C+1}{2}\left(\frac{t^3}{6}-\frac{t^2}{2}\right),
\]
using $x_1(0)=0$. Imposing $x_1(1)=1$ yields $C=-7$, and therefore
\begin{align}
u^*(t) &= 3-3t,\\
x_2^*(t) &= 3t-\frac{3}{2}t^2,\\
x_1^*(t) &= \frac{3}{2}t^2-\frac{1}{2}t^3.
\end{align}

\begin{remark}
The essential point is structural: with $n$ states one introduces $n$ adjoints. The dimension of the two-point boundary value problem doubles relative to the state dimension.
\end{remark}

\subsection{State Elimination and Problem Reduction}

In multi-state models it is often possible to remove superfluous state variables before writing the adjoint system.

\subsubsection{Structural Elimination}

Suppose the problem is
\[
\max_{u(\cdot)} \int_{t_0}^{t_1} f(t,x_1,x_2,u)\,dt
\]
subject to
\begin{align}
\dot{x}_1 &= g_1(t,x_1,x_2,u),\\
\dot{x}_2 &= g_2(t,x_1,x_2,u),\\
\dot{x}_3 &= g_3(t,x_1,x_2,x_3,u),
\end{align}
with given initial data. If neither the payoff nor the first two state equations depend on $x_3$, then the pair $(x_1,x_2)$ with control $u$ forms a closed optimal control problem. One solves for $(x_1^*,x_2^*,u^*)$ first and then recovers $x_3^*$ from its forward ODE.

\subsubsection{Elimination via Explicit Integration}

A second mechanism occurs when one state can be solved explicitly and substituted into the objective.

\begin{example}[Bioreactor-style reduction]
Suppose
\begin{align}
\dot{x}(t) &= \frac{G\,u(t)}{H+u(t)}x(t)-D x(t)^2,\\
\dot{z}(t) &= -Kx(t)z(t),
\end{align}
with objective
\[
\min_{u(\cdot)} \left\{ \ln z(T) + \int_0^T A\,u(t)\,dt \right\}.
\]
The second state solves explicitly as
\[
z(t)=z_0 \exp\!\left(-K\int_0^t x(s)\,ds\right),
\]
hence
\[
\ln z(T)=\ln z_0 - K\int_0^T x(s)\,ds.
\]
Since $\ln z_0$ is constant, the problem is equivalent to
\[
\max_{u(\cdot)} \int_0^T \left(Kx(t)-A u(t)\right)\,dt
\]
subject only to the scalar state equation for $x$. This removes the state $z$ and also removes one adjoint variable from the necessary conditions.
\end{example}

\subsection{Finite-Horizon Continuous-Time LQR}

We now specialize to the linear quadratic regulator, one of the few classes for which a nearly closed-form structural solution exists.

\subsubsection{Problem Statement}

Let
\begin{equation}
\dot{x}(t)=A(t)x(t)+B(t)u(t), \qquad x(0)=x_0,
\end{equation}
where
\[
x(t)\in\R^n,\qquad u(t)\in\R^m,
\]
and $A(t)\in\R^{n\times n}$, $B(t)\in\R^{n\times m}$.

The finite-horizon quadratic objective is
\begin{equation}
J(u)=\frac{1}{2}x(T)^\top M x(T)+\frac{1}{2}\int_0^T
\left( x(t)^\top Q(t)x(t)+u(t)^\top R(t)u(t)\right)\,dt,
\end{equation}
where
\[
M=M^\top \succeq 0,\qquad Q(t)=Q(t)^\top \succeq 0,\qquad R(t)=R(t)^\top \succ 0.
\]
The positive definiteness of $R(t)$ ensures invertibility.

\subsubsection{Hamiltonian Derivation}

The Hamiltonian is
\[
H=\frac{1}{2}x^\top Qx+\frac{1}{2}u^\top Ru+\lambda^\top(Ax+Bu).
\]
The stationarity condition is
\[
0=\pdv{H}{u}=Ru+B^\top \lambda,
\]
hence
\begin{equation}
u^*(t)=-R(t)^{-1}B(t)^\top \lambda(t).
\end{equation}
The adjoint equation is
\begin{equation}
\dot{\lambda}(t)=-Q(t)x(t)-A(t)^\top \lambda(t),
\qquad
\lambda(T)=Mx(T).
\end{equation}

\subsubsection{Riccati Ansatz}

Assume
\begin{equation}
\lambda(t)=S(t)x(t),
\end{equation}
where $S(t)\in\R^{n\times n}$ is to be determined. Differentiating and substituting the state and adjoint equations yields the Riccati differential equation
\begin{equation}
-\dot{S}(t)=A(t)^\top S(t)+S(t)A(t)-S(t)B(t)R(t)^{-1}B(t)^\top S(t)+Q(t),
\end{equation}
with terminal condition
\begin{equation}
S(T)=M.
\end{equation}

\begin{theorem}[Finite-horizon LQR feedback law]
If $S(t)$ solves the Riccati equation above, then the optimal control is
\begin{equation}
u^*(t)=-R(t)^{-1}B(t)^\top S(t)x(t).
\end{equation}
Thus the optimal control is a linear state feedback, with gain matrix
\[
K(t)=R(t)^{-1}B(t)^\top S(t).
\]
\end{theorem}

The Riccati equation trades a forward-backward Hamiltonian system for a backward matrix ODE. After solving $S(t)$ backward from $T$ to $0$, one integrates the closed-loop state equation
\[
\dot{x}(t)=\left(A(t)-B(t)R(t)^{-1}B(t)^\top S(t)\right)x(t).
\]

\subsubsection{Scalar Example}

Consider
\[
\min_{u(\cdot)} \frac{1}{2}\int_0^T \left(x(t)^2+u(t)^2\right)\,dt
\]
subject to
\[
\dot{x}(t)=u(t), \qquad x(0)=x_0.
\]
Here $A=0$, $B=1$, $Q=1$, $R=1$, and $M=0$, so the Riccati equation becomes
\[
-\dot{S}=1-S^2,\qquad S(T)=0.
\]
This scalar ODE integrates to
\[
S(t)=\frac{1-e^{2(t-T)}}{1+e^{2(t-T)}}.
\]
The optimal feedback law is
\[
u^*(t)=-S(t)x(t),
\]
and therefore the optimal state solves
\[
\dot{x}(t)=-S(t)x(t).
\]
A direct integration yields
\[
x^*(t)=x_0\,\frac{e^t+e^{2T-t}}{1+e^{2T}},
\qquad
u^*(t)=x_0\,\frac{e^t-e^{2T-t}}{1+e^{2T}}.
\]

\subsection{Higher-Order Differential Equations as Systems}

Pontryagin’s principle is formulated for first-order systems. Higher-order controlled differential equations must therefore be lifted to state-space form.

\subsubsection{General Reduction}

Suppose the problem is
\[
\max_{u_1,\dots,u_m} \int_{t_0}^{t_1}
f\bigl(t,x,\dot{x},\dots,x^{(n)},u_1,\dots,u_m\bigr)\,dt
\]
subject to
\[
x^{(n+1)}(t)=g\bigl(t,x,\dot{x},\dots,x^{(n)},u_1,\dots,u_m\bigr),
\]
with initial conditions
\[
x(t_0)=\alpha_1,\quad \dot{x}(t_0)=\alpha_2,\quad \dots,\quad x^{(n)}(t_0)=\alpha_{n+1}.
\]

Define
\[
x_1=x,\qquad x_2=\dot{x},\qquad \dots,\qquad x_{n+1}=x^{(n)}.
\]
Then the system becomes
\begin{align}
\dot{x}_1 &= x_2,\\
\dot{x}_2 &= x_3,\\
&\vdots \\
\dot{x}_n &= x_{n+1},\\
\dot{x}_{n+1} &= g(t,x_1,\dots,x_{n+1},u_1,\dots,u_m),
\end{align}
with the original objective now rewritten in terms of $(x_1,\dots,x_{n+1})$.

\subsubsection{Second-Order Example}

Consider
\[
\min_{u(\cdot)} \frac{1}{2}\int_0^\pi \left(u(t)^2-x(t)^2\right)\,dt
\]
subject to
\[
x''(t)=u(t), \qquad x(0)=1,\quad x'(0)=1.
\]
Set
\[
x_1=x,\qquad x_2=x'.
\]
Then
\begin{align}
\dot{x}_1 &= x_2,\qquad x_1(0)=1,\\
\dot{x}_2 &= u,\qquad x_2(0)=1.
\end{align}
The Hamiltonian is
\[
H=\frac{1}{2}u^2-\frac{1}{2}x_1^2+\lambda_1 x_2+\lambda_2 u.
\]
The first-order conditions are
\begin{align}
0 &= \pdv{H}{u}=u+\lambda_2 \quad \Longrightarrow \quad u^*=-\lambda_2,\\
\dot{\lambda}_1 &= -\pdv{H}{x_1}=x_1,\\
\dot{\lambda}_2 &= -\pdv{H}{x_2}=-\lambda_1,
\end{align}
with terminal conditions $\lambda_1(\pi)=\lambda_2(\pi)=0$ if the terminal state is free and there is no terminal payoff.

This example is a useful template: every higher-order optimal control problem can be converted to a first-order Hamiltonian system of larger dimension.

\subsection{Isoperimetric Constraints}

Integral constraints on the trajectory or the control are common in applications: fixed total fuel, fixed cumulative dosage, fixed resource expenditure, or a prescribed total harvested biomass.

\subsubsection{General Form}

Consider
\begin{equation}
\max_{u(\cdot)} \left\{
\int_{t_0}^{t_1} f(t,x(t),u(t))\,dt + \phi(x(t_1))
\right\}
\end{equation}
subject to
\begin{align}
\dot{x}(t) &= g(t,x(t),u(t)), \qquad x(t_0)=x_0,\\
\int_{t_0}^{t_1} h(t,x(t),u(t))\,dt &= B,\\
a &\le u(t)\le b.
\end{align}

\subsubsection{Auxiliary-State Conversion}

Define the additional state
\[
z(t)=\int_{t_0}^{t} h(s,x(s),u(s))\,ds.
\]
Then
\[
\dot{z}(t)=h(t,x(t),u(t)), \qquad z(t_0)=0,\qquad z(t_1)=B.
\]
Thus the isoperimetric constraint is converted into an endpoint constraint, and the problem becomes a standard multi-state problem.

\begin{remark}
This conversion is conceptually simple but extremely important. It is the standard way to bring integral equality constraints into the scope of Pontryagin’s principle without introducing an external Lagrange multiplier at the functional level.
\end{remark}

\subsubsection{Example}

Consider
\[
\min_{u(\cdot)} \frac{1}{2}\int_0^1 u(t)^2\,dt
\]
subject to
\begin{align}
\dot{x}(t)&=u(t),\qquad x(0)=0,\quad x(1)=1,\\
\int_0^1 x(t)\,dt&=2.
\end{align}

Introduce
\[
\dot{z}(t)=x(t),\qquad z(0)=0,\quad z(1)=2.
\]
The Hamiltonian becomes
\[
H=\frac{1}{2}u^2+\lambda_1 u+\lambda_2 x.
\]
Adjoint equations:
\begin{align}
\dot{\lambda}_1 &= -\pdv{H}{x}=-\lambda_2,\\
\dot{\lambda}_2 &= -\pdv{H}{z}=0.
\end{align}
Hence $\lambda_2=C$ and
\[
\lambda_1(t)=k-Ct.
\]
Stationarity gives
\[
0=\pdv{H}{u}=u+\lambda_1
\quad \Longrightarrow \quad
u^*(t)=Ct-k.
\]
Integrating twice,
\begin{align}
x(t)&=\frac{C}{2}t^2-kt,\\
z(t)&=\frac{C}{6}t^3-\frac{k}{2}t^2.
\end{align}
Using $x(1)=1$ and $z(1)=2$, one finds
\[
C=-18,\qquad k=-10.
\]
Therefore
\begin{align}
u^*(t)&=10-18t,\\
x^*(t)&=10t-9t^2.
\end{align}

\subsection{Advanced remarks for PhD readers}

\subsubsection{On Existence Versus Necessary Conditions}

Pontryagin’s principle provides necessary conditions, not existence. Existence typically requires separate arguments involving:
\begin{enumerate}[label=(\roman*)]
    \item compactness of admissible controls,
    \item convexity/coercivity of the running cost,
    \item growth bounds for the dynamics,
    \item lower semicontinuity of the objective.
\end{enumerate}
In nonlinear biological and engineering models, failure of convexity can create multiple extremals, singular arcs, or nonexistence in the classical sense.

\subsubsection{Second-order issues}
When the Hamiltonian is affine in the control and switching functions have high-order zeros, first-order conditions may be inconclusive. Singular arc analysis, generalized Legendre--Clebsch conditions, and second variation arguments become necessary.

\subsubsection{Relationship to HJB}
The Hamiltonian maximization condition in PMP matches the optimization appearing in the Hamilton--Jacobi--Bellman equation:
\[
-V_t(t,x)=\max_{u\in U}\big\{f(t,x,u)+\nabla_x V(t,x)^\top g(t,x,u)\big\}.
\]
Formally, the adjoint plays the role $\lambda(t)=\nabla_x V(t,x^*(t))$ along an optimal trajectory when the value function is sufficiently smooth.

\subsubsection{Singular Controls}

When the Hamiltonian is affine in a control component,
\[
H(\cdot)=\cdots + \psi(t)u,
\]
the stationarity condition does not determine $u$ in the interior. One must then differentiate the switching function $\psi$ until the control appears explicitly. Singular arc analysis lies beyond the scope of these notes, but students should recognize the warning sign: a control enters linearly and boundedly, yet the switching function vanishes on an interval.

\subsubsection{Numerics}

General nonlinear multi-state optimal control problems are typically solved numerically by one of three paradigms:
\begin{enumerate}[label=(\alph*)]
    \item \textbf{Indirect methods:} solve the state-adjoint boundary value problem.
    \item \textbf{Direct methods:} discretize the control/state and solve a nonlinear program.
    \item \textbf{Dynamic programming / HJB methods:} feasible only in low dimension because of the curse of dimensionality.
\end{enumerate}
LQR is exceptional because the Riccati equation yields an analytically structured feedback solution.

\subsubsection{Take-home principles}

\begin{enumerate}[label=(\roman*)]
    \item The state and adjoint equations keep their Hamiltonian form under box constraints.
    \item The control law changes from \emph{stationarity} to \emph{pointwise optimization over the admissible set}.
    \item For scalar controls, sign conditions on $H_u$ determine whether the optimizer lies at the lower bound, upper bound, or in the interior.
    \item For quadratic Hamiltonians, the optimal constrained control is the projection of the unconstrained optimizer onto the admissible interval.
    \item Bang-bang structure emerges when the Hamiltonian is affine in the control.
    \item One must enforce bounds during the solve; clipping an unconstrained solution afterward is generally wrong.
    \item In indirect numerics, bounded controls are implemented by a projected control update inside each forward-backward sweep iteration.
\end{enumerate}

\section{Bibliographic remarks and Acknowledgment}
Please see the relevant references: \cite{king_graded_1982,
cohen_maximizing_1971,
barbu_analysis_1993,
barbu_mathematical_1994,
li_optimal_1995,
liu_bidirectional_2025,
jacobson_extensions_1980,
yong_stochastic_1999,
wang_multi-strategy_2025,
wade_introduction_2010,
mooney_course_1999,
edelstein-keshet_mathematical_2005,
jones_differential_2009,
kot_elements_2001,
murray_mathematical_2003,
pontryagin_mathematical_2018,
yu_pattern_2026,
rudin_real_1987,
stein_real_2005,
wang_analysis_2025,
clarke_optimization_1990,
fleming_deterministic_1975,
gao_rolling_2022,
kamien_dynamic_2003,
macki_introduction_1982,
cesari_optimizationtheory_1983,
lukes_differential_1982,
filippov_certain_1962,
edelstein-keshet_mathematical_1991,
berkovitz_optimal_1976,
sethi_optimal_2000,
lewis_optimal_1986,
burden_numerical_1997,
wang_algebraicspectral_2026,
de_feo_optimal_2024,
wang_path_2022,
federico_linear-quadratic_2025,
lin_controlled_2020,
wang_global_2021,
bayen_hybrid_2024,
wang_linear-quadratic_2023,
huang_infinite_2025,
dou_time-inconsistent_2020,
liu_characterizations_2022,
liang_global_2025,
wang_spike_2023,
trelat_exponential_2025,
calvia_state_2021,
hasenohr_computer-assisted_2025,
wang_analysis_2025-1,
conforti_kl_2025,
szpruch_optimal_2024,
li_neural_2024,
archibald_numerical_2024,
tang_exploratory_2022,
giegrich_convergence_2024,
han_deep_2026,
liu_hilbert_2025,
cai_soc-martnet_2025,
mayorga_finite_2023,
wang_damage-structured_2026,
zhou_policy_2025,
meng_general_2025,
lew_sample_2024,
huang_sublinear_2025,
wang_stochastic_2025}.

\bibliographystyle{unsrtnat}
\bibliography{references}

\end{document}